\documentclass[a4paper,10pt]{article}
\usepackage{amsmath,amssymb,amsthm,graphicx,wrapfig,pstricks,fullpage}
\usepackage[numbers]{natbib}

\newtheorem{theorem}{Theorem}[section]

\numberwithin{equation}{section}

\begin{document}

\title{\vspace{-30pt}An introduction to the trapping experienced by biased random walk on the trace of biased random walk}
\author{D.~A.~Croydon\footnote{Department of Advanced Mathematical Sciences, Graduate School of Informatics, Kyoto University, Sakyo-ku, Kyoto 606--8501, Japan. \texttt{croydon@acs.i.kyoto-u.ac.jp}}}
\maketitle

\begin{abstract} We introduce and summarise results from the recent paper `Biased random walk on the trace of biased random walk on the trace of\dots' \cite{CH}, which was written jointly with M.~P.~Holmes (University of Melbourne). We also present additional discussion on some of the conjectures made in \cite{CH}. The content of this article is loosely based on the presentation given by the author at the Probability Symposium held at the Research Institute for Mathematical Sciences, Kyoto University in December 2018.
\end{abstract}

\section{Introduction}

Consider the random walk $X=(X_n)_{n\geq 0}$ on the integer lattice $\mathbb{Z}^d$ that is $\beta$-times more likely to jump in first-coordinate direction than in any other direction, i.e.\ its transition probabilities are proportional to the following weights (we will usually, but not always, assume $\beta>1$):
\begin{center}
\scalebox{0.3}{\includegraphics{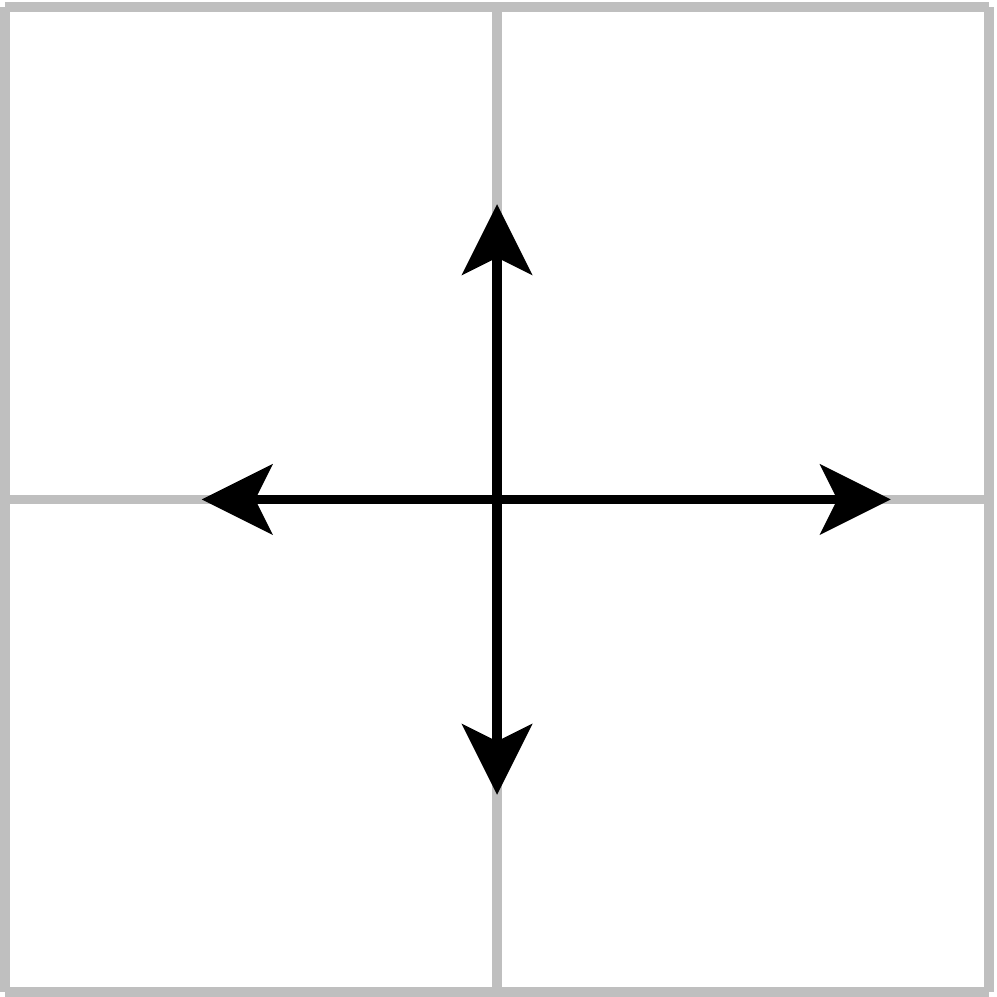}}
\put(-14,32){$\beta$}
\put(-70,32){$1$}
\put(-50,67){$1$}
\put(-50,12){$1$}
\end{center}
For this simple model, elementary results yield that at time $n$ the position of the walk is described as follows:
\[X_n = \frac{(\beta-1)n {e}_1}{\beta+2d-1}  + c_{\beta,d} N(0,I)n^{1/2}+o(n^{1/2}),\]
where $e_1$ is the unit vector in the first coordinate direction, $c_{\beta,d}$ is a constant depending on $\beta$ and $d$, $N(0,I)$ is a standard $d$-dimensional Gaussian vector, and the final term is $o(n^{1/2})$ in probability. (Of course, there are plenty of more refined statements that could be made, including functional ones.) In particular, as soon as the bias $\beta$ is taken to be non-trivial, then the process $X$ moves on a linear, or ballistic scale. Indeed, one might further note from the law of large numbers that, almost-surely,
\[\frac{X_n}{n}\rightarrow \frac{(\beta-1)e_1}{\beta+2d-1}.\]
As one would expect, this shows the walk has asymptotic velocity in the direction $e_1$, and its speed $v(\beta)=(\beta-1)/(\beta+2d-1)$ is monotonic in $\beta$, increasing from $0$ at $\beta=1$ to $1$ as $\beta\rightarrow\infty$.

From the point of view of mathematical physics, a natural question to ask is: How is the behaviour described in the previous paragraph affected by the introduction of disorder into the medium? In other words, what can we say if the underlying graph $\mathbb{Z}^d$ is replaced by a more irregular one? In the 1980s, physicists realised that for random graphs such as percolation clusters (about which we provide further discussion in Subsection \ref{percsec}) the answer to these questions would depend on the interplay between two effects resulting from an increase in drift. On the one hand, in certain `nice' parts of the environment the effect of additional drift would be as for the Euclidean lattice, enabling the walk to head in its direction of transience more quickly. On the other hand, there might also be sections of the environment, one might say `traps', such that if the random walk entered one of these, then a larger drift would make it more difficult to escape. This intuition was supported using heuristic arguments, via which it was suggested for biased random walks on percolation clusters there would be non-monotonicity of the speed as the bias increased, and sub-ballisticity (zero speed) in the strong bias regime \cite{BD}.

Mathematically, a phase transition between ballisticity and sub-ballisticity was first shown rigorously for the simpler model of biased random walk on supercritical Galton-Watson trees, and has since been confirmed in the supercritical percolation setting. In Section \ref{backgroundsec}, we briefly review the known results in these areas; a much more extensive survey is given in \cite{BAF}. Here and in \cite{BAF}, there is also discussion of biased random walk on critical Galton-Watson trees, studies of which are partially motivated by the aim of understanding the corresponding process on a critical percolation cluster. As we will discuss briefly in Subsection \ref{critpercsec} below, from the latter model there moreover arises an interest in studying biased random walks on random paths. Earlier results (surveyed in Subsection \ref{rworwsec} below) treated the case where the underlying random path had no preferred direction; in \cite{CH}, introducing the results of which is the focus of this article, attention is placed on the case when the underlying random path is directionally transient -- this is the model of `biased random walk on the trace of biased random walk (BRWBRW)' of the title. See Section \ref{brwbrwsec} for the part of the article discussing this model. We also highlight that for biased random walk on random paths, there is a strong connection with the model of one-dimensional random walk in random environment, and in Subsection \ref{conjsec} we use this parallel to formulate conjectures about more precise aspects of the behaviour of BRWBRW.

\section{Background}\label{backgroundsec}

In this section, we present a brief overview of previous results for biased random walks on random trees, percolation clusters and random paths.

\subsection{Supercritical (and subcritical) Galton-Watson trees}\label{suptreesec}

The easiest case in which to describe the trapping of biased random walk on a random graph is for a random tree. The case of supercritical Galton-Watson trees in particular was first explored in \cite{LPP}, and we start by briefly describing the developments for this model, as well as for the related subcritical model.

A Galton-Watson tree is a generated by a branching process: an initial ancestor gives birth to offspring according to an offspring distribution $(p_n)_{n\geq 0}$ to form the first generation, and then, given a particular generation, particles in this each give birth independently according to $(p_n)_{n\geq 0}$ to form the subsequent generation. In the supercritical case, corresponding to the offspring distribution having mean $m>1$, the process survives for all time with non-zero probability -- we condition on this event to give the underlying graph of interest in the next part of the discussion.

We then consider the random walk $X=(X_n)_{n\geq 0}$ which jumps along nearest neighbour edges, with probability of jumping to a particular neighbour away from the root (initial ancestor) being $\beta$-times more likely than jumping towards the root (when both are possible). For this process, it is clear to see that on the `backbone' -- the collection of vertices with a direct path to infinity, increasing the bias $\beta$ helps the walk escape. However, if $p_0>0$, then there exist dead-ends in the environment, and increasing the bias makes it more difficult to escape from these. The basic structure is shown in the left-hand figure below, and the right-hand figure shows the conjectured random walk asymptotic speed $v(\beta)=\lim_{n\rightarrow\infty}|X_n|/n$, where $|x|$ is the generation of $x$ (both figures are sourced from \cite{BAF}).
\begin{center}
\begin{tabular}{cc}
{\scalebox{0.3}{\includegraphics{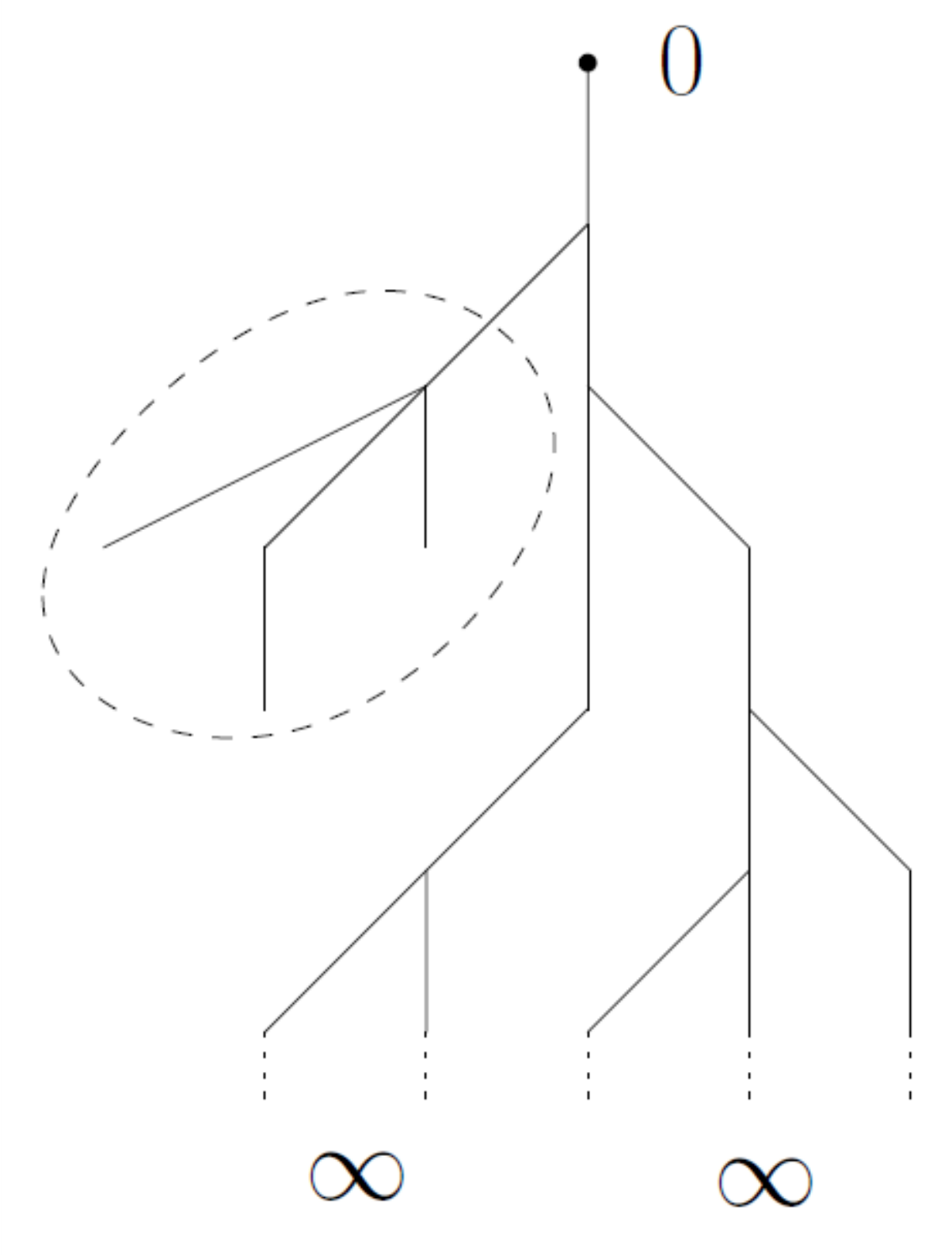}}}& \begin{tabular}{c}\vspace{-140pt} \hphantom{hhh} \\\scalebox{0.28}{\includegraphics{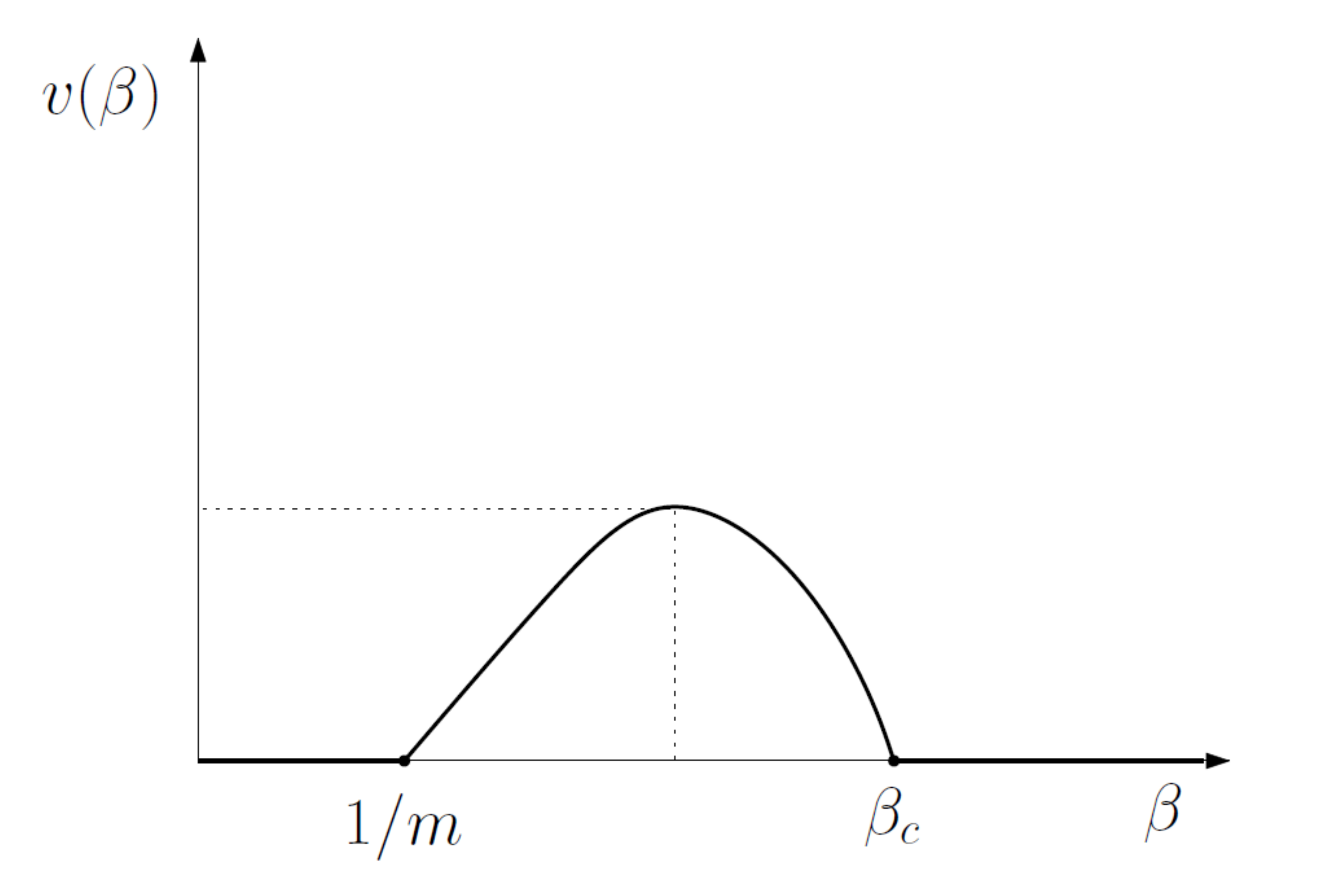}}\end{tabular}
\end{tabular}
\end{center}
Whilst the particular form of the speed is still unproven, the ballistic/sub-ballistic phase transition was established in \cite{LPP}, with $\beta_c$ being identified explicitly as $1/m_{trap}$, where $m_{trap}$ is the expected number of offspring of vertices in traps. (NB. For low biases, $\beta<1/m$, the walk is recurrent.) See also \cite{BFGH} for recent work concerning more detailed properties of $X$ in the sub-ballistic, $\beta>\beta_c$ regime, and \cite{Bowditch} for work on the related subcritical ($m<1$) model. In particular, the following schematic phase diagram sourced from \cite{Bowditch} illustrates how the exponent seen as the $n\rightarrow\infty$ limit of $\log|X_n|/\log n$ varies with the parameters of the model. (Note this figure shows $\mu$ in place of $m$.) Whilst we will not fully explain the figure here, the key quantities are: the bias $\beta$; $m_{trap}$ say, which is shown as $f'(q)$ in the below figure for the supercritical case, and is given by $m$ in the subcritical case; and the polynomial exponent $\alpha$ describing the tail regularity of the offspring distribution (specifically, meaning that it is chosen to fall into the domain of attraction of an $\alpha$-stable random variable). The quantity $\gamma$ is given by $-\log(m_{trap})/\log(\beta)$.
\begin{center}
{\scalebox{0.4}{\includegraphics{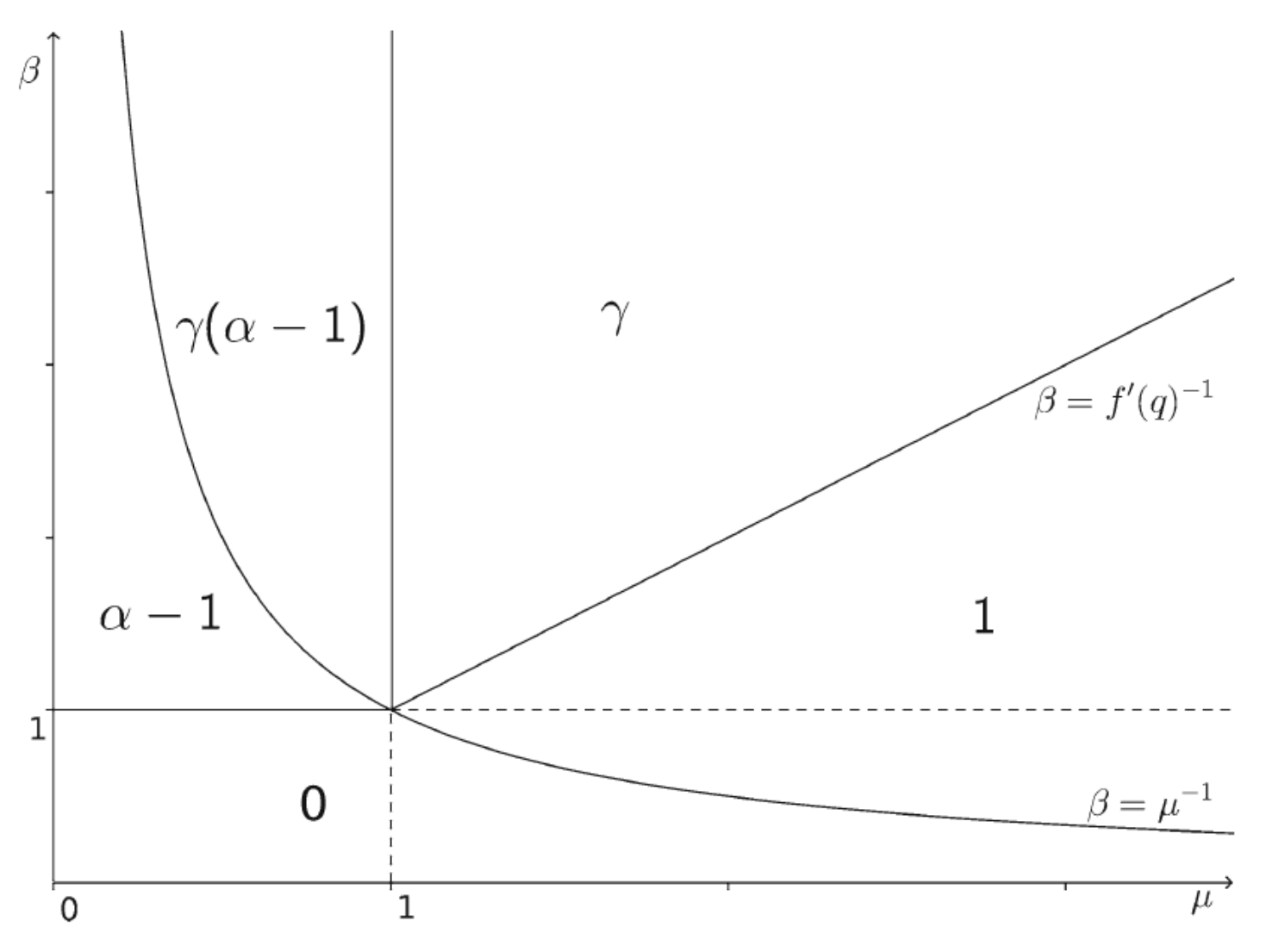}}}
\end{center}

\subsection{Supercritical percolation}\label{percsec}

A more challenging model is that of percolation, though the ballistic/sub-ballistic phase transition predicted by physicists has now been established rigourously, see \cite{BGP, FH, Sznitman}. To describe this more precisely, we recall that bond percolation on the integer lattice $\mathbb{Z}^d$ ($d\geq 2$) is the process whereby nearest neighbour edges are independently retained with some probability $p$, and discarded otherwise. Above the critical threshold for an infinite cluster to exist, one naturally asks about the asymptotic behaviour of the (biased) random walk on this graph (which is unique):
\vspace{-30pt}
\begin{center}
\reflectbox{\scalebox{0.4}{\includegraphics{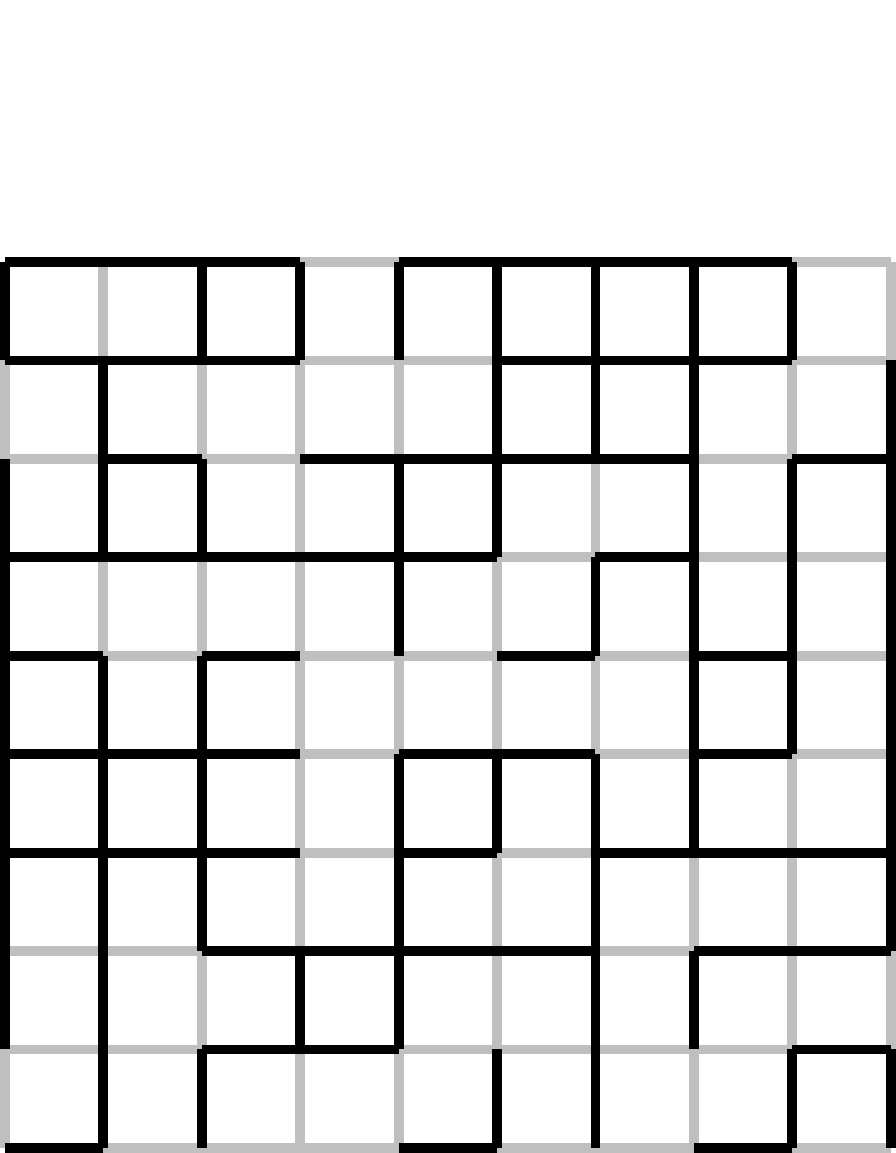}}}
\end{center}
As in the Euclidean setting described at the start of the article, we suppose the associated biased random walk jumps along edges of the infinite cluster, but is $\beta$-times more likely to jump in the first coordinate direction than in any other direction, when this is a possibility. If $\beta=1$, then the random walk is behaves as it does on the whole of $\mathbb{Z}^d$, in that it is has Gaussian fluctuations for almost-every realisation of the environment \cite{BB,MP,SS}. Moreover, if $\beta>1$, then the walk is directionally transient. However, in this latter case there exists a $\beta_c\in(1,\infty)$ such that if $\beta<\beta_c$, then the biased random walk has positive speed, but if $\beta>\beta_c$, then the biased random walk has zero speed \cite{BGP,FH,Sznitman}.
As in the case of Galton-Watson trees, confirming this ballistic/sub-ballistic phase transition depended on providing a careful analysis of the dead-ends in the environment from which the walk has to back-track to escape.

\subsection{Critical percolation}\label{critpercsec}

In early work concerning biased random walk on percolation close to criticality, physicists identified two types of trapping that might occur \cite{BD}:
  \begin{center}
   \scalebox{0.45}{\includegraphics{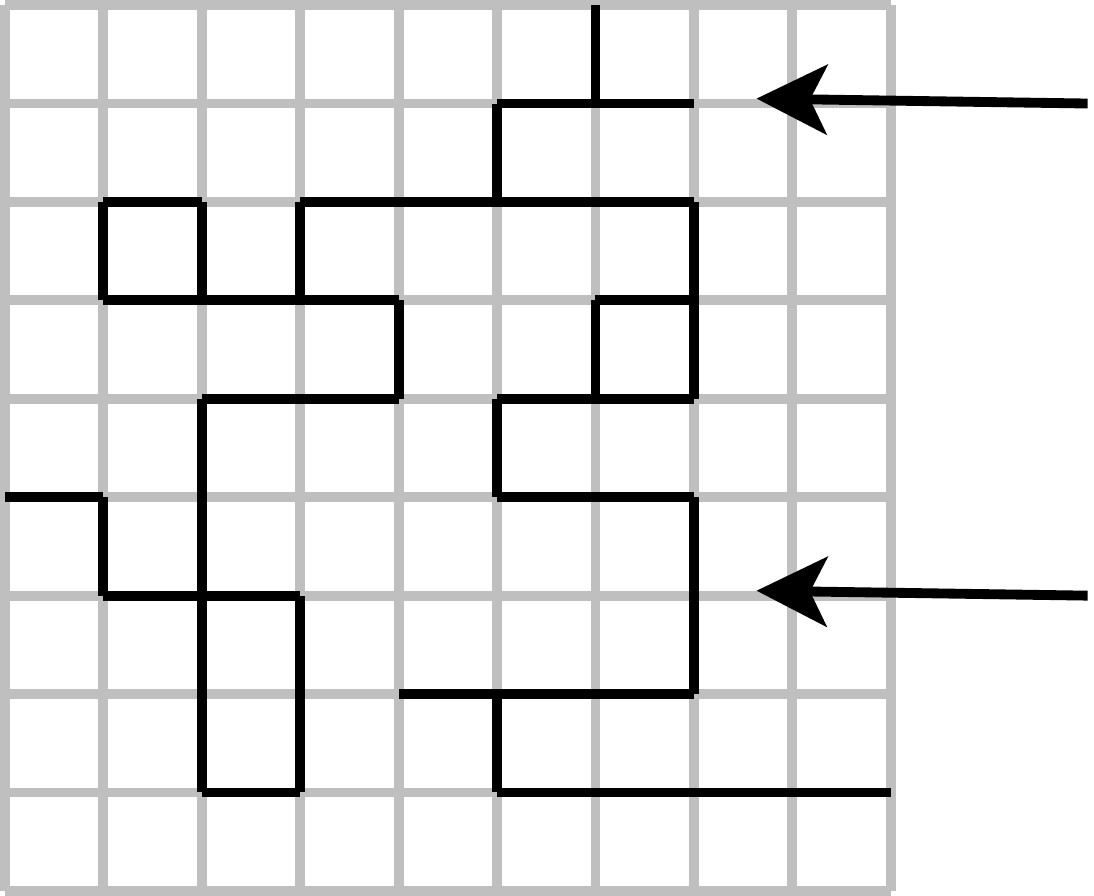}}
  \end{center}

\vspace{-122pt}
\hspace{285pt}`Trapping in branches'

\vspace{50pt}
\hspace{285pt}`Traps along the backbone'

\vspace{40pt}

\noindent
In particular, trapping in branches captures the idea of dead-ends already described in the previous two subsections. The additional phenomena of traps along the backbone results from the relative sparseness of paths in near critical percolation clusters, as well as their spatial shape, which means that the biased random will spend longer in certain parts of the path than others. Of course, both types of trap ask that the random walk fights against the bias, and to understand the biased random walk's asymptotic properties, one must estimate the time it takes to overcome these difficult regions. Whilst the case of biased random walk on the critical percolation cluster is still unexplored rigourously, the above considerations do suggest an interest in the following models for biased random walk:
\begin{itemize}
  \item biased random walk on a critical Galton-Watson tree conditioned to survive, so as to capture the tree-like structure of high-dimensional critical percolation clusters (as originally identified in \cite{HS});
  \item biased random walk on the trace of a random walk in high dimensions, so as to capture the effect of trapping along the backbone, which is known to scale like a simple symmetric random walk (in this high dimensional case), see \cite{HHHM}.
\end{itemize}
These two models are discussed in the next two subsections, and the second also partially motivates the study of biased random walk on a biased random walk, as is pursued in the subsequent section. We further note some interesting recent work regarding biased random walk on spatially embedded critical trees from \cite{Andrio}, in which a possible scaling limit for biased random walk on a critical percolation cluster in high dimensions in a weak bias regime is identified. The latter work also suggests an approach for studying localisation properties of biased random walk on critical percolation in high dimensions for more general biases.

\subsection{Critical Galton-Watson trees}

For a Galton-Watson tree, as introduced in Subsection \ref{suptreesec}, with an offspring distribution that is critical, i.e.\ with mean $m=1$, and which is in the domain of attraction of $\alpha$-stable random variable, $\alpha\in(1,2]$, the associated biased random walk $X=(X_n)_{n\geq 0}$ was studied in \cite{CFK}. In this regime, the tree is almost-surely finite, but there is a well-understood limiting procedure via which one can condition it to survive for all time. With this as the underlying graph, it is possible to show that
\[\left(\frac{\left|\pi(X_{e^{nt}})\right|\log\beta}{(\alpha-1)n}\right)_{t\geq 0}\]
converges in distribution to a certain stochastic process that is independent of the parameters in the model, where $\pi(x)$ is the closest backbone vertex to $x$ (i.e.\ the result ignores the depth of the current trap). Not only does the result makes clear the dependence of $\alpha$ and $\beta$ on the scaling, but also shows the biased random walk $X$ is always sub-ballistic for any $\beta>1$, and in fact moves extremely slowly -- taking an exponential amount of time to escape from a ball (cf.\ the supercritical and subcritical cases, where the analogous escape rate is polynomial). The key computation in this case is checking that the height of a trap $H$ satisfies $\mathbf{P}\left(H\geq x\right)\sim\frac{1}{(\alpha-1)x}$, rather than the exponential tail seen in the non-critical cases. Indeed, the result then follows by noting that the escape time behaves like $\beta^H$ and thus has slowly varying tail $\frac{\log\beta}{(\alpha-1)\log x}$, and so the time to escape along the backbone, which behaves like an i.i.d.\ sum of such random variables, grows exponentially with distance. In fact, the extremal nature of the trapping distribution means that there is also a strong localisation effect in the deep traps, meaning that it is possible to predict with high certainty where the walk is at a particular time.

\subsection{Trace of symmetric random walk in high dimensions}\label{rworwsec}

As the final piece of background, we consider biased random walk on the trace of symmetric random walk. Specifically, let $X^{(0)}$ be simple symmetric random walk on $\mathbb{Z}^d$, $d\geq 5$, which we recall is a transient process. Let $\mathcal{G}^{(0)}=(V^{(0)},E^{(0)})$ be graph with vertex set and edge set given by
\begin{equation}\label{v0e0}
V^{(0)}:=\left\{X^{(0)}_n:n \in \mathbb{Z}_+\right\},\qquad E^{(0)}:=\left\{\left\{X^{(0)}_n,X^{(0)}_{n+1}\right\}:\:n \in \mathbb{Z}_+\right\},
\end{equation}
respectively. Conditional on $\mathcal{G}^{(0)}$, we then let $X^{(1)}$ be the $\beta$-biased random walk on this graph, defined similarly to the biased random walk on a percolation cluster. As in the previous section, the trapping in this case is very strong, and it is possible to check the following localisation result:
\[\mathbf{P}\left(\left|\frac{X^{(1)}_n}{\log n}-L_n\right|>\varepsilon\right)\rightarrow 0,\]
where $L_n$ is a $\mathcal{G}^{(0)}$-measurable random variable that converges in distribution to a random variable $L_{\beta}$, and the distribution of $\log(\beta)L_\beta$ is independent of $\beta$. The argument is quite different in flavour to the tree/percolation cases in that it does not involve comparing with an i.i.d.\ sum, but rather using techniques developed for one-dimensional random walk in a random environment (see \cite{Zeitouni} for background). In particular, it turns out that the process behaves essentially like a random walk in a random potential, where the random potential is closely approximated by $-\log(\beta)(X^{(0)}\cdot e_1)$. It is known from the one-dimensional case that in the regime under consideration, the walk spends most of its time in `valleys' of the potential, whereby to escape from a valley of depth $h$ it will take a time $\beta^h$ to escape. An example of a valley that leads to the localisation result presented above is shown in the following sketch (with the bias direction $e_1$ pointing upwards).
\begin{center}
   \scalebox{0.45}{\includegraphics{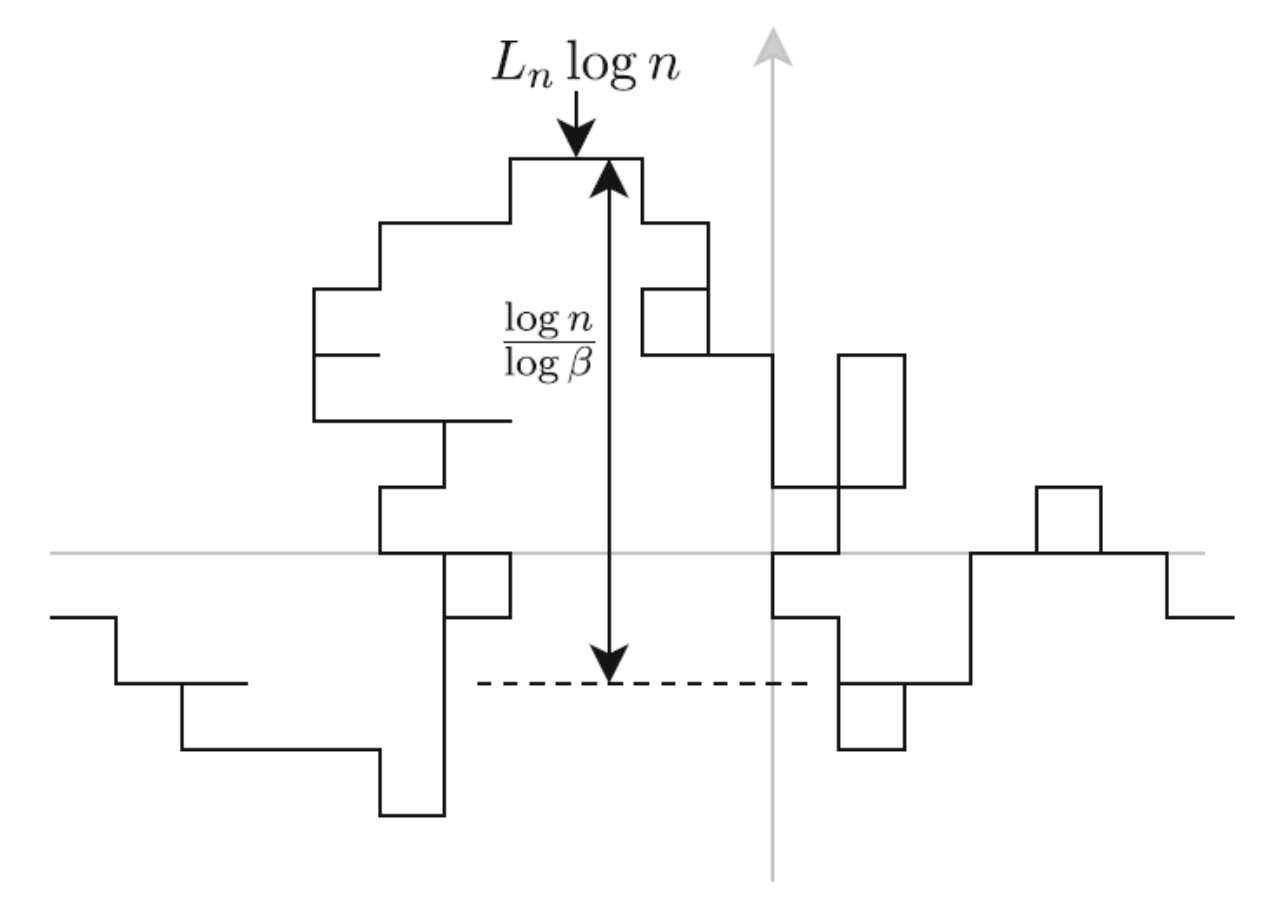}}
  \end{center}

\section{Biased random walk on the trace of biased random walk}\label{brwbrwsec}

We now come to the main focus of this article -- the biased random walk on the trace of biased random walk (BRWBRW). This model is introduced precisely in Subsection \ref{modelsec}. In Subsections \ref{rectransec} and \ref{ballsubsec} we present our results concerning criteria for recurrence/transience and for ballisticity/sub-ballisticity. These are qualitatively similar to those known to hold for Galton-Watson trees and percolation clusters, and perhaps the main interest is in the explicit description of how trapping arises for this model. Finally, in Subsection \ref{conjsec}, we conclude the article with some conjectures concerning more detailed behaviour than is proven in \cite{CH}.

\subsection{The model}\label{modelsec}

The underlying graph of interest in this section will be generated by a random walk $X^{(0)}=(X^{(0)}_n)_{n\geq 0}$ on $\mathbb{Z}^d$ that has transition distribution $\mathbf{p}^{(0)}$ supported on the standard basis vectors $\{\pm e_j:j \in \{1,2,\dots,d\}\}$. We always suppose that $X^{(0)}_0=0$, and that the process has a drift with strictly positive first coordinate, i.e.\
\[\delta^{(0)}:=\sum_{e}e\mathbf{p}^{(0)}(e),\]
satisfies $\delta^{(0)}\cdot e_1>0$. (Only that the drift is non-zero, and not the particular direction of drift, is important in what follows.) We then let $\mathcal{G}^{(0)}=(V^{(0)},E^{(0)})$ be the graph with vertex set and edge set given by \eqref{v0e0}; this is the trace of biased random walk. Conditional on $\mathcal{G}^{(0)}$, we then suppose $X^{(1)}$ has transition probabilities
\[P^{\mathcal{G}^{(0)}}\left(X^{(1)}_{n+1}=x+e\:|\:X^{(1)}_n=x\right)=
\frac{\mathbf{p}^{(1)}(e)}{\sum_{e':(x,x+e')\in E^{(0)}}\mathbf{p}^{(1)}(e')}\]
for $(x,x+e)\in E^{(0)}$, where $\mathbf{p}^{(1)}$ is also a probability measure supported on the standard basis vectors; $X^{(1)}$ is the BRWBRW. The following sketch shows an initial part of the trace of $X^{(1)}$ (heavier lines) on $\mathcal{G}^{(0)}$ (lighter lines).
\begin{center}
   \scalebox{1.2}{\includegraphics{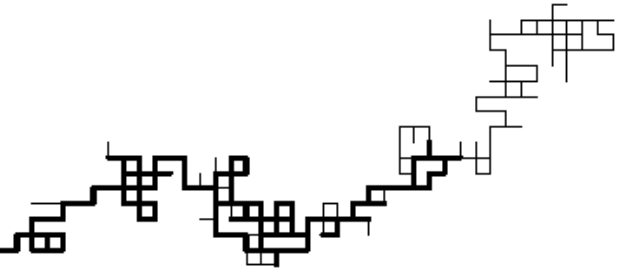}}
\end{center}
Importantly, to ensure that $X^{(1)}$ is well-defined, we assume that $p^{(1)}(e)>0$ for all $e \in \{\pm e_j:j \in \{1,2,\dots,d\}\}$. This further means that the transition probabilities of $X^{(1)}$ can alternatively be written:
\[P^{\mathcal{G}^{(0)}}\left(X^{(1)}_{n+1}=x+e\:|\:X^{(1)}_n=x\right)=\frac{c^{(1)}(x,x+e)}{\sum_{e':(x,x+e')\in E^{(0)}}c^{(1)}(x,x+e')},\]
for $(x,x+e)\in E^{(0)}$, where `edge conductances' are given by
\begin{equation}\label{conddef}
c^{(1)}(x,y)=\left(\prod_{j=1}^d \mathbf{p}^{(1)}(-e_j)^{|y_j-x_j|}\right)\beta_{(1)}^{(x\vee y)\cdot\ell^{(1)}}
\end{equation}
with $\ell^{(1)}$ the unit vector parallel to, and $\log\beta^{(1)}$ the norm of
\[\left(\log\left(\frac{\mathbf{p}^{(1)}(e_j)}{\mathbf{p}^{(1)}(-e_j)}\right)\right)_{j=1}^d.\]
NB. $\beta_{(1)}\geq1$, with equality only in the balanced case, i.e.\ when $\mathbf{p}^{(1)}(e_j)=\mathbf{p}^{(1)}(-e_j)$ for all $j\in\{1,2,\dots,d\}$.

\subsection{Recurrence/transience}\label{rectransec}

The drift associated with $\mathbf{p}^{(1)}$ on $\mathbb{Z}^d$ is given by
\[\delta^{(1)}:=\sum_{e}e\mathbf{p}^{(1)}(e),\]
and it might be a natural first guess that if this is oriented in the same direction as the drift associated with $\mathbf{p}^{(0)}$, i.e.\ $\delta^{(0)}\cdot\delta^{(1)}>0$, then $X^{(1)}$ would be transient (drift to infinity with probability one), and that it would be recurrent (return to zero infinitely often with probability one) otherwise. However, it turns out this is not the right criteria for recurrence/transience, as the next result illustrates. In particular, one should consider instead the direction of `conductance drift' $\ell^{(1)}$, as defined in the previous section.

\begin{theorem} The graph $\mathcal{G}^{(0)}$ is almost-surely transient for $X^{(1)}$ if $\delta^{(0)}\cdot\ell^{(1)}>0$, and is almost-surely recurrent for $X^{(1)}$ otherwise.
\end{theorem}
\begin{proof}[Proof sketch] Using the fundamental connections between electrical networks and random walks (see, for example, \cite{LP}), recurrence is equivalent to having an unbounded resistance between 0 and infinity in the graph $\mathcal{G}^{(0)}$, when this is equipped with edge conductances as at \eqref{conddef}. Given the path-like nature of the graph, it is possible to check that the resistance to infinity is of the same order as the sum of edge resistances, i.e.\ the order of
\[\sum_{n=0}^{\infty}\beta_{(1)}^{-X_n^{(0)}\cdot\ell^{(1)}}\approx \sum_{n=0}^{\infty}\beta_{(1)}^{-n\delta^{(0)}\cdot\ell^{(1)}},\]
where the second approximation is obtained from the law of large numbers for $X^{(0)}$. The result follows.
\end{proof}

To confirm it is not possible to replace $\ell^{(1)}$ by $\delta^{(1)}$ in the previous result, we note that it is possible to find examples where the drifts are oriented as follows:
\begin{center}
\scalebox{0.5}{\scalebox{0.7}{\scalebox{0.6}{\includegraphics{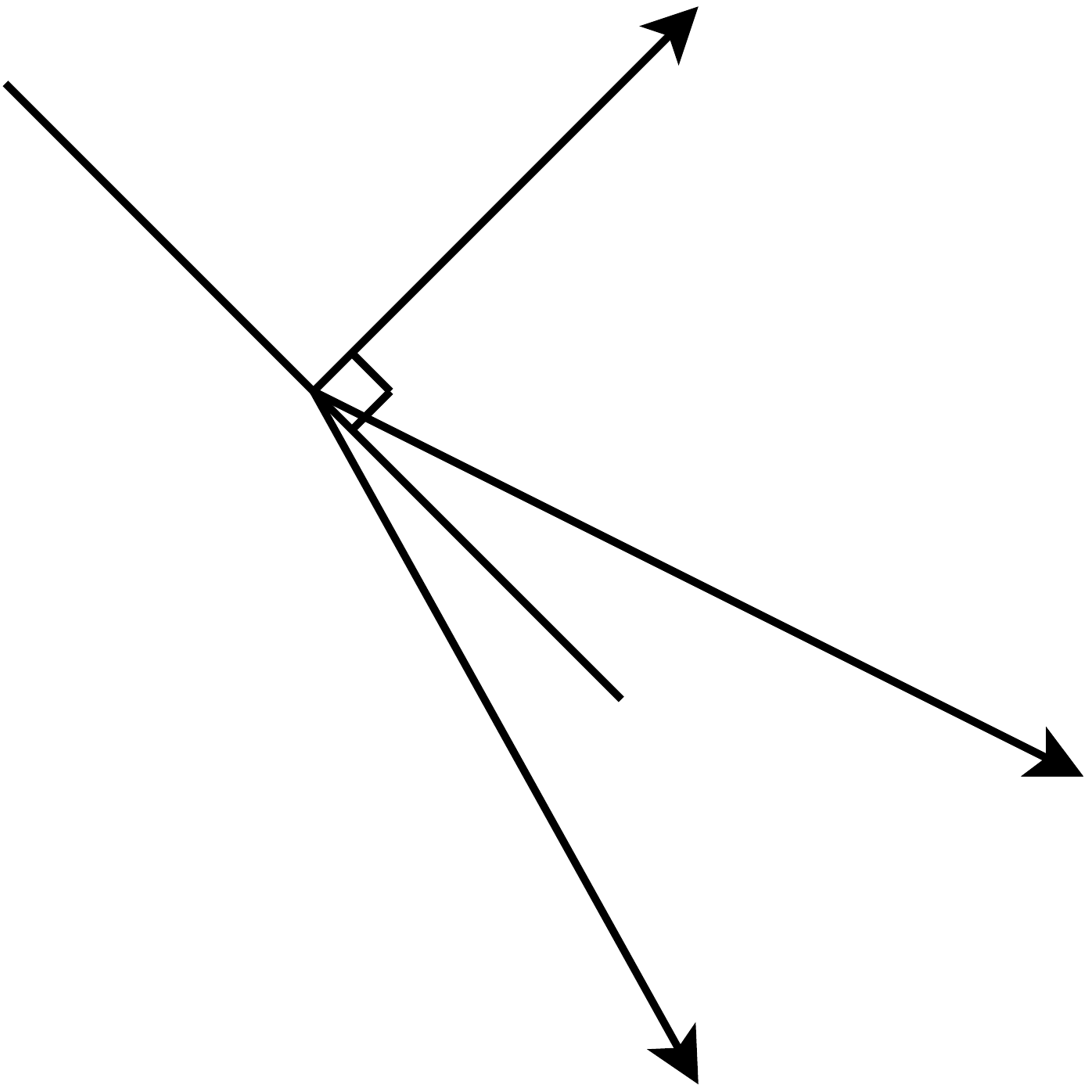}}
\put(-80,240){\scalebox{2}{$\delta^{(0)}$}}
\put(-80,0){\scalebox{2}{$\delta^{(1)}$}}
\put(10,65){\scalebox{2}{$\ell^{(1)}$}}}}
\end{center}
So, the walk $X{(1)}$ is transient (and can even be ballistic), even though $\delta^{(0)}\cdot\delta^{(1)}<0$. And, likewise, one can also see the opposite effect, whereby the walk $X{(1)}$ is recurrent, even through $\delta^{(0)}\cdot\delta^{(1)}>0$. We note these effects even hold for a deterministic path, and similarly counterintuitive results have also been observed for one-dimensional random walk in random environment (see \cite{Solomon}).

\subsection{Ballistic/Sub-ballistic phase transition}\label{ballsubsec}

We next describe the main results of \cite{CH}, which concern the ballistic/sub-ballistic phase transition for $X^{(1)}$. Assuming we are in the transient regime, i.e.\ $\delta^{(0)}\cdot\ell^{(1)}>0$, we are able to prove the following. The proof of part (a) depends on standard regeneration arguments, and so the main novelty is in establishing part (b).

\begin{theorem}\label{subbthm} (a) [Limiting velocities] Almost-surely, the following limit exists:
\[v^{(1)}:=\lim_{n \rightarrow \infty} n^{-1}X^{(1)}_n=v\delta^{(0)}.\]
for some deterministic $v\in[0,\infty)$.\\
(b) [Ballistic/Sub-ballistic phase transition] There exists an $\alpha_{(1)}\in(1,\infty)$ such that:\\
(i) If $\beta_{(1)}<\alpha_{(1)}$, then $X^{(1)}$ is ballistic, i.e.\ $v^{(1)}\neq 0$.\\
(ii) If $\beta_{(1)}>\alpha_{(1)}$,  then $X^{(1)}$ is sub-ballistic, i.e.\ $v^{(1)}=0$.
\end{theorem}

In the proof of this result, not only are we able to explicitly characterise $\alpha_{(1)}$, but we are also able to describe the structure of traps that lead to sub-ballisticity. Concerning the definition of $\alpha_{(1)}$, first let
\[\varphi_{(1)}(t)=\mathbf{E}\left[\exp\{-tX^{(0)}_1\cdot \ell^{(1)}\}\right]\]
be the moment generating function of $-X^{(0)}_1\cdot \ell^{(1)}$. Under the assumption that $\delta^{(0)}\cdot\ell^{(1)}>0$, basic properties of such a function yields that there is a unique strictly positive solution $t_{(1)}$ to $\varphi_{(1)}(t)=1$, and we set
\[\alpha_{(1)}=\exp\{t_{(1)}\}.\]
The importance of this quantity to our story arises from its alternative description in terms of the back-tracking exponent for $X^{(0)}$ in the direction $\ell^{(1)}$, i.e.
\begin{equation}\label{backtrack}
-x^{-1}\log\mathbf{P}\left(\min_nX^{(0)}_n\cdot \ell^{(1)}\leq -x\right)\rightarrow\log \alpha_{(1)}.
\end{equation}
In particular, similarly to what is known to be the case in the Galton-Watson tree and percolation settings, one would expect that if $X^{(0)}_n\cdot \ell^{(1)}$ back-tracks a distance $h$, then this will create a trap that takes time $\beta_{(1)}^h$ to escape from. As a result, the tail of the escape time from the trap created at 0 will behave approximately like
\begin{equation}\label{tail}
\mathbf{P}\left(\beta_{(1)}^{-\min_nX^{(0)}_n\cdot \ell^{(1)}}\geq t\right)\approx t^{-\frac{\log\alpha_{(1)}}{\log\beta_{(1)}}}.
\end{equation}
Moreover, standard regeneration arguments allow us to show that traps arise in certain appropriately-chosen sections of  $\mathcal{G}^{(0)}$ (of finite expected length) in an i.i.d.\ fashion. Putting these observations together, we arrive at the criteria for ballisticity being that $\beta_{(1)}^{-\min_nX^{(0)}_n\cdot \ell^{(1)}}$ has a first moment, and thus that $\log\alpha_{(1)}/\log\beta_{(1)}>1$. (We exclude the boundary case from consideration, since this requires a more careful argument that was not pursued in \cite{CH}.)

To make the previous argument rigourous, an important step is to show that back-tracking actually leads to the creation of traps, for self-intersections of $X^{(0)}$ mean the former effect does not automatically lead to the latter. Specifically, in \cite{CH}, it was considered whether one might see sections of the environment of the following form for some suitably-chosen vector $\hat{\delta}^{(1)}$.
\begin{center}
\scalebox{0.72}{
\includegraphics[scale=0.46]{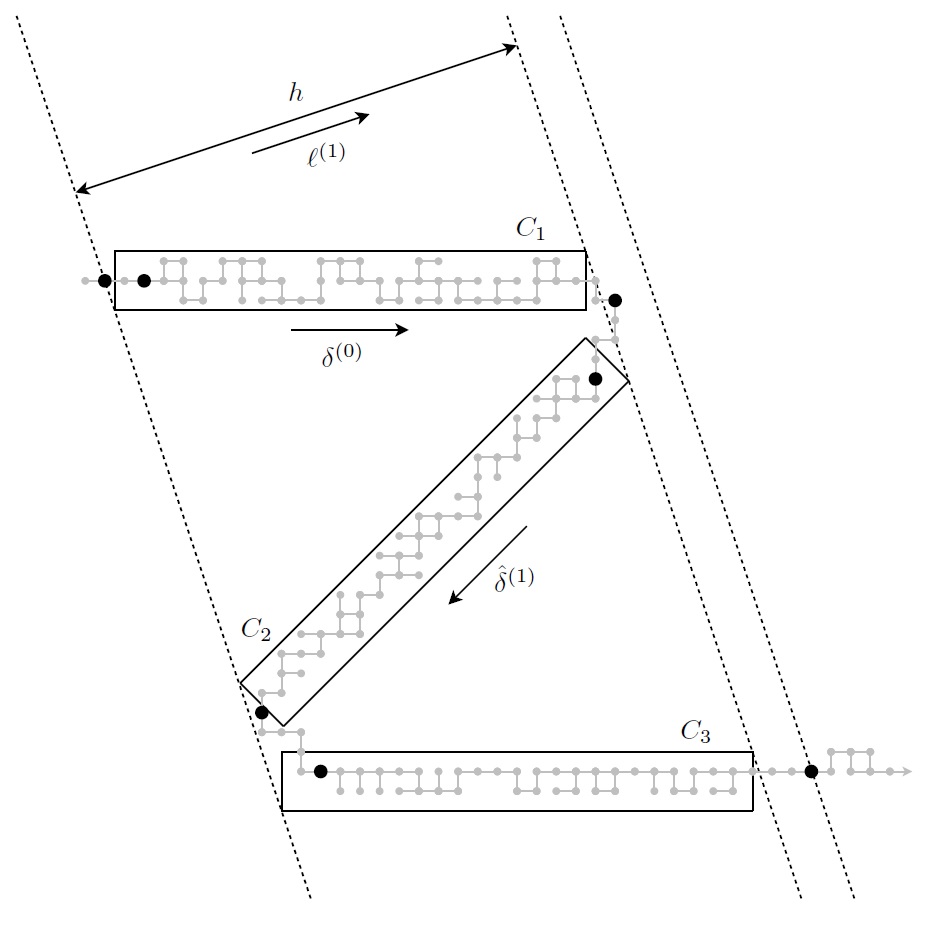}}
\end{center}
In particular, it is being asked that $X^{(0)}$ moves inside a thin cylinder $C_1$ in the direction $\delta^{(0)}$ until it has increased its $\ell^{(1)}$ coordinate by $h$, it then moves inside a thin cylinder $C_2$ in a direction $\hat{\delta}^{(1)}$ until it has decreased its $\ell^{(1)}$ coordinate by $h$, and finally, it moves inside a thin cylinder $C_3$ in the direction $\delta^{(0)}$ until it has again increased its $\ell^{(1)}$ coordinate by $h$. Of course, this event implies that $X^{(0)}\cdot\ell^{(1)}$ contains a back-tracking section of size $h$, and the question becomes whether it is possible to choose $\hat{\delta}^{(1)}$ so that the probability of the event behaves as at \eqref{backtrack}. It transpires this is possible, with the least costly direction to take being:
\[\hat{\delta}^{(1)}:=\mathbf{E}\left(\alpha_{(1)}^{-X_1^{(0)}\cdot\ell^{(1)}}X_1^{(0)}\right).\]
Intuition for this choice is that it means $\hat{\delta}^{(1)}$ is the drift of the Markov process with transition probabilities given by the `Doob transform'
\[\hat{p}^{(0)}(y-x)=\frac{{p}^{(0)}(y-x)h^{(1)}(y)}{h^{(1)}(x)},\]
where $h^{(1)}$ is the $X^{(0)}$-harmonic function given by $h^{(1)}(x):=\alpha_{(1)}^{-x\cdot\ell^{(i)}}$. Roughly speaking, this is the law of $X^{(0)}$ conditioned so that $X^{(0)}_n\cdot \ell^{(1)}\rightarrow-\infty$ (cf. \cite[Section 17.6]{LPW}). To make the argument rigourous, a classic large deviations principle of Mogul$'$ski\u\i\:\cite{Mogulskii} was applied.

We finish the section with an example to illustrate the transparency of the criteria of Theorem \ref{subbthm}. Suppose that for $i\in\{0,1\}$ there exists a $k_i\in\{1,2,\dots,d\}$ and $\gamma_i>1$ such that
\[\mathbf{p}^{(i)}(e)= \frac{1+(\gamma_i-1) \mathbf{1}_{e\in\{e_1,\dots,e_{k_i}\}}}{2d+k_i(\gamma_i-1)}.\]
We then have that $X^{(1)}$ is ballistic if $k_1(\gamma_1-1)<(k_1\wedge k_0)(\gamma_0-1)$, and sub-ballistic if the reverse (strict) inequality is true. In particular, if $k_1\leq k_0$, then the condition simplifies further to $\gamma_1<\gamma_0$.

\subsection{Conjecture}\label{conjsec}

Towards conjecturing further detail concerning the behaviour of $X^{(1)}$, we start by noting that if we observe $X^{(1)}$ at the cut-points of $\mathcal{G}^{(0)}$ (i.e.\ vertices that, when removed, disconnect $\mathcal{G}^{(0)}$), then we obtain a random walk in a one-dimensional random environment. Similarly to what is seen in the model of \cite{C2}, as discussed in Subsection \ref{rworwsec} above, the potential of this walk should essentially behave like
\[-\log(\beta_{(1)})\left(X^{(0)}\cdot\ell^{(1)}\right),\]
and so $X^{(1)}\cdot\delta^{(0)}$ should behave as per a one-dimensional random walk with a similar potential. In particular, comparing with the known results for independent and identically-distributed one-dimensional random environments \cite{ESZ,KKS} (see also \cite{ESTZ}), a key parameter for determining the behaviour of  $X^{(1)}$ is likely to be
\[\kappa_{(1)}:=\frac{\log\alpha_{(1)}}{\log\beta_{(1)}},\]
which we already met in the tail estimate at \eqref{tail}. Indeed, we would then expect (modulo a caveat raised below) that: if $\kappa_{(1)}>2$, then
\[\frac{\left(X_n^{(1)}-nv^{(1)}\right)\cdot\delta^{(0)}}{\sqrt{n}}\]
will converge in distribution to a (non-trivial) Gaussian random variable; if $\kappa_{(1)}\in(1,2)$, then
\[\frac{\left(X_n^{(1)}-nv^{(1)}\right)\cdot\delta^{(0)}}{{n}^{1/\kappa_{(1)}}}\]
will converge in distribution to a completely asymmetric stable zero mean random variable of index $\kappa_{(1)}$; if $\kappa_{(1)}\in(0,1)$, then
\[\frac{X_n^{(1)}\cdot\delta^{(0)}}{{n}^{\kappa_{(1)}}}\]
will converge in distribution to the inverse of a power of a completely asymmetric stable zero mean random variable of index $\kappa_{(1)}$; and, moreover in the boundary cases $\kappa_{(1)}=1,2$, the behaviour of $X^{(1)}\cdot\delta^{(0)}$ should also match that of the corresponding one-dimensional random walk in random environment. The one additional subtlety that could arise for the BRWBRW is that there will potentially be a lattice effect if $\ell^{(1)}$ has rational coordinates (cf.\ a similar issue in \cite{BFGH} and also \cite{Bowditch}). In this case, the scaling exponents could be expected to remain the same, but convergence might only occur subsequentially, with a modified limiting distribution.

\setlength{\bibsep}{0pt plus 0.3ex}
\footnotesize{
\bibliography{rimsproc}
\bibliographystyle{amsplain}
}

\end{document}